\documentclass{amsart}
\textwidth =150mm\usepackage{amsmath}
\usepackage{amsfonts}
\usepackage{amssymb}
\usepackage{mathrsfs}
\usepackage{stmaryrd}
\usepackage{bbm}
\usepackage{oldgerm}
\usepackage[english]{babel}
\usepackage[T1]{fontenc}
\usepackage[latin1]{inputenc}
\usepackage[all]{xy}
\usepackage{hyperref}

\newtheorem{theorem}[subsection]{Theorem}
\newtheorem{proposition}[subsection]{Proposition}
\newtheorem{corollary}[subsection]{Corollary}
\newtheorem{lemma}[subsection]{Lemma}

\theoremstyle{definition}

\newtheorem{remark}[subsection]{Remark}

\numberwithin{equation}{subsection}

\usepackage{amsmath}
\usepackage{amsthm}
\usepackage{mathrsfs}
\usepackage{amsfonts}

\begin{document}

\title {On Katz's $(A,B)$-exponential sums}

\author{Lei Fu and Daqing Wan}
\address{Yau Mathematical Sciences Center, Tsinghua University, Beijing 100084, P. R. China} 
\email{leifu@math.tsinghua.edu.cn}
\address{Department of Mathematics, University of California, Irvine, CA 92697, USA}
\email{dwan@math.uci.edu}

\date{}
\maketitle

\begin{abstract}
We deduce Katz's theorems for $(A,B)$-exponential sums over finite fields using $\ell$-adic cohomology and 
a theorem of Denef-Loeser,  removing the hypothesis that $A+B$ is relatively prime to the characteristic $p$.  In some degenerate cases, 
the Betti number estimate is improved using  toric decomposition and Adolphson-Sperber's bound for the degree of 
$L$-functions.  Applying the facial decomposition theorem in \cite{W1}, we prove that the universal 
family of $(A,B)$-polynomials is generically ordinary for its $L$-function when $p$ is in 
certain arithmetic progression. 
\end{abstract} 

\section*{Introduction}
Let $k$ be a finite field with $q$ elements of characteristic $p$ and let $d$ be a positive integer. 
A polynomial $f(t_1, \ldots, t_n)$ in $k[x_1,..., t_n]$ of degree $d>0$ is called a \emph{Deligne polynomial} 
 if $d$ is prime to $p$, and the leading form $f_d$ of $f$ defines a smooth projective hypersurface $f_d=0$ in ${\mathbb P}^{n-1}$. 
For positive integers $A$ and $B$, the $(A,B)$-polynomial considered by Katz is a Laurent polynomial of the following form: 
\begin{equation}\label{G}
G(t_0,\ldots, t_n):=t_0^A f(t_1, \ldots, t_n) +g(t_1, \ldots, t_n)+P_B(1/t_0) \in k[t_0^{\pm}, t_1,..., t_n]
\end{equation}
where $f(t_1, \ldots, t_n)$ is a Deligne polynomial of degree $d$, $g(t_1, \ldots, t_n)$ is a polynomial of degree $<d$, 
and $P_B(s)$ is a one-variable polynomial of degree $\leq B$. It can be viewed as a one-parameter 
family of Deligne polynomials parametrized by $t_0$ over the torus ${\mathbb G}_m$. We stress that for flexibility of applications, 
the degree of $P_B$ is only assumed to be at most $B$.  Later on we distinguish the two cases when the degree 
is exactly $B$ and when the degree is strictly less than $B$. 

Fix a nontrivial additive character $\psi: (k,+)\to {\mathbb C}^*$. For a Deligne polynomial $f(t_1,..., t_n)$ as above, 
one has Deligne's fundamental estimate (\cite[3.7.2.1]{D})
$$\big| \sum_{t\in k^n} \psi(f(t)) \big| \leq (d-1)^n q^{n/2}.$$
Define the $L$-function by 
$$L(f, T) = \exp\left(\sum_{m=1}^{\infty} \frac{T^m}{m}\sum_{t\in k_m^n} \psi\Big({\rm Tr}_{k_m/k}(f(t)\Big)\right),$$
where $k_m$ denotes the extension of $k$ of degree $m$.
Deligne shows that $L(f, T)^{(-1)^{n-1}}$ is a polynomial of degree $(d-1)^n$ pure of weight $n$. 
From the $p$-adic point of view, Sperber \cite{S} further shows that the $q$-adic Newton polygon of the polynomial $L(f, T)^{(-1)^{n-1}}$
lies above a certain lower bound called the Hodge polygon, which is defined to be the $q$-adic Newton polygon of the polynomial 
$$h(T)=\prod_{1\leq j_1, \ldots, j_n\leq d-1} (1- q^{\frac{j_1}{d}+\cdots + \frac{j_n}{d}}T).$$
These two polygons coincide 
for a generic Deligne polynomial $f$ of degree $d$ over $\bar{k}$ if $p\equiv 1 \mod d$, that is, the universal 
family of Deligne polynomials of degree $d$ in $n$ variables is generically ordinary for its 
$L$-function when $p\equiv 1 \mod d$. 
This is because  for the diagonal polynomial 
$$f(t)= t_1^d +\ldots + d_n^d,$$
the classical Stickelberger theorem for the Gauss sum implies that the Newton polygon equals to the Hodge polygon 
if $p\equiv 1 \mod d$. 

Motivated by applications in analytic number theory, it is of interest to study the exponential sum for the 
above $(A,B)$-polynomial $G(t_0, \ldots, t_n)$. If one applies Deligne's estimate fibre by fibre to the exponential sum 
for the $(A, B)$-polynomial, one gets the 
``trivial'' bound 
$$\Big \vert \sum_{t_0\in k^*}\sum_{t_1,\ldots, t_n\in k} \psi\Big(G(t_0,\ldots, t_n)\Big)\Big\vert
\leq (q-1) (d-1)^n q^{n/2} \leq (d-1)^n q^{(n+2)/2}.$$ 
This is already of considerable depth, but is still weaker than the expected square root cancellation as $q$ varies. 
 Using $\ell$-adic cohomology, Katz \cite{K} proves the following optimal 
square root estimate for the $(A,B)$-exponential sum 
over ${\mathbb G}_m \times {\mathbb A}^n$. 

\begin{theorem}\label{I} Suppose that $f$ is a Deligne polynomial of degree $d$ prime to $p$, 
$p$ is prime to $AB$, $\deg(P_B)=B$ and $\mathrm{deg}(g)<\frac{Bd}{A+B}$. 
For the $(A,B)$-polynomial $G(t_0,\ldots, t_n)$ defined in equation (\ref{G}), we have the estimate 
$$\Big \vert \sum_{t_0\in k^*}\sum_{t_1,\ldots, t_n\in k} \psi\Big(G(t_0,\ldots, t_n)\Big)\Big\vert
\leq (A+B)(d-1)^n q^{(n+1)/2}.$$ 
\end{theorem}

By a standard reduction procedure, we can always reduce the exponential sum to the case where $AB$ is not divisible by $p$, unless $P_B(t_0)$ is reduced to a constant.  
Thus the condition that $p$ is prime to $AB$ is not essential. 
The condition $\mathrm{deg}(g)<\frac{Bd}{A+B}$ is necessary to ensure that the Betti number is bounded by $(A+B)(d-1)^n$. 
There is an extra condition that $p$ is prime to $A+B$ in Katz's 
original theorem. In this paper, we give a proof of the above theorem using a theorem of Denef-Loeser, removing this 
extra assumption that $p$ is prime to $A+B$.  

If $\deg(P_B)=B$ and $f$ is affine Dwork regular (see section $2$ for its definition), 
the relevant cohomology is pure (Katz's $(A,B)$ purity theorem). 
In this case, we check that the $(A,B)$-polynomial $G(t_0,\ldots, t_n)$ is non-degenerate and commode
(with respect to $t_1, \ldots, t_n$) so that we can apply the theorem of Denef-Leoser 
to deduce the purity and to calculate the exact Betti number.  
If $\deg(P_B)=B$ but $f$ is only assumed to be a Deligne polynomial,  the $(A,B)$-polynomial $G(t_0,\ldots, t_n)$ may not be non-degenerate  
and Denef-Loeser's theorem may not apply. In this case, the shifted polynomial $f(x) +a$ is affine Dwork 
regular for most $a\in \bar{k}$. Following Katz, we use a specialization and perverse argument to show that the same estimate 
remains true.  The relevant cohomology in the degenerate case is mixed, and the number $(A+B)(d-1)^n$ is 
only an upper bound for the Betti number. 

Assume $\mathrm{deg}(P_B)=h\leq B$. Theorem \ref{I} applies only when $\deg(g)=e< hd/(A+h)$ and we get 
$$\Big \vert \sum_{t_0\in k^*}\sum_{t_1,\ldots, t_n\in k} \psi\Big(G(t_0,\ldots, t_n)\Big)\Big\vert
\leq (A+h)(d-1)^n q^{(n+1)/2}.$$ 
It would be interesting to extend the theorem to the case when $dh/(A+h) \leq e <d$. 
This cannot be done in general, see \cite[Remark 5.5]{K}. However, 
Katz has a trick to make it work  if $f(x)$ is affine 
Dwork regular. His idea is to choose a larger $B$ so that $e< Bd/(A+B)$ and 
consider the family $P_B(t_0) + bt_0^B$ parametrized by $b\in \bar{k}$.  
Using a similar specialization and perverse argument, 
the same theorem can still be proved. 
We obtain the following theorem, again proved first by Katz under the extra assumption that $p$ is relatively prime to 
$A+B$. 

\begin{theorem}\label{II} Suppose that $f$ is affine Dwork regular of degree $d$ prime to $p$, 
$p$ is prime to $AB$, $\deg(P_B)\leq B$ and $\mathrm{deg}(g)=e<\frac{Bd}{A+B}$. Then we have 
$$\Big \vert \sum_{t_0\in k^*}\sum_{t_1,\ldots, t_n\in k} \psi\Big(G(t_0,\ldots, t_n)\Big)\Big\vert
\leq (A+B)(d-1)^n q^{(n+1)/2}.$$ 
\end{theorem}

As mentioned above, if $\deg(P_B)=h$ and $e\geq hd/(A +h)$, then Theorem \ref{I} does not apply. 
We can choose a larger $B\geq h$ so that $B$ is prime to $p$ and $e <Bd/(A+B)$. Then Theorem \ref{II} 
will apply with this larger $B$ at the expense of increasing the constant $A+h$ to $A+B$.     
For fixed $A$, there are many choices of $(A,B)$ such that $e< Bd/(A+B)$ and $h\leq B$. We want to 
choose such $B$ prime to $p$ such that $A+B$ (equivalently $B$) is as small as possible. 

Consider the case 
$h\leq 1$, that is, $P_B$ is at most 
a linear polynomial (possibly a constant).  
For each positive integer $A$, the smallest positive integer $B$ satisfying $e< Bd/(A+B)$ is the smallest 
integer $B$ that is greater than $eA/(d-e)$. Thus the coefficient in the estimate of the above 
theorem can be taken to be 
$$\Big(A+\Big[ \frac{eA}{d-e}\Big]+1\Big)(d-1)^n.$$ 

We can improve the Betti number estimate $(A+B)(d-1)^n$ of Theorem \ref{II} in the case $e\geq hd/(A +h)$ 
(equivalently $h\leq eA/(d-e)$) as follows. 
The case $e< hd/(A +h)$ is handled by Theorem \ref{I} and the Betti number estimate is optimal already. 

\begin{theorem}\label{III} Suppose that $f$ is affine Dwork regular of degree $d$ prime to $p$,
$p$ is prime to $A$, $\deg(P_B)=h$, $\deg(g)=e$, and $d> e\geq hd/(A +h)$. Then we have 
$$\Big \vert \sum_{t_0\in k^*}\sum_{t_1,\ldots, t_n\in k} \psi\Big(G(t_0,\ldots, t_n)\Big)\Big\vert
\leq \Big( \Big(A+ \frac{eA}{d-e}\Big)(d+1)^{n} -\Big (\frac{eA}{d-e}-h\Big)(e+1)^{n}\Big)q^{(n+1)/2}.$$ 
\end{theorem}

Taking $P_B=0$, we obtain 

\begin{corollary}\label{IV} Suppose that $f(t)$ is affine Dwork regular of degree $d$ prime to $p$, 
and $p$ is prime to $A$.  For any polynomial $g(t)\in k[t_1,..., t_n]$ of degree $\deg(g)=e<d$, 
we have $$\Big \vert \sum_{t_0\in k^*}\sum_{t_1,\ldots, t_n\in k} \psi\Big(t_0^A f(t_1, \ldots, t_n) +g(t_1, \ldots, t_n)\Big)\Big\vert
\leq \Big(\Big (A+ \frac{eA}{d-e}\Big)(d+1)^{n} - \frac{eA}{d-e}(e+1)^{n}\Big)q^{(n+1)/2}.$$ 
\end{corollary}

As indicated above, if $e< Bd/(A+B)$, the coefficient in this error term can be smaller than 
the coefficient $(A+B)(d-1)^n$ in Theorem \ref{II}. This indicates that the estimate in Theorem \ref{III} 
is generally better than the estimate in Theorem \ref{II} in the case $e\geq hd/(A +h)$.

We remark that this corollary cannot be deduced from the theorem of Denef-Loeser as the leading form $g_e$ of the polynomial $g(t_1,\ldots, t_n)$ 
can be highly singular so that the Laurent polynomial $G(t_0,\ldots, t_n)=t_0^Af(t_1,\ldots, t_n) +g(t_1,\ldots, t_n)$ 
can be highly degenerate. Our proof of Theorem \ref{III} combines the cohomological consequence 
that the $L$-function or its reciprocal is a polynomial together with toric decomposition and Adolphson-Sperber's bound for the degree of
the $L$-functions of toric exponential sums. 

In this paper, we also study the generic Newton polygon for the $L$-function associated to $(A,B)$-exponential sums. 
A lower bound ${\rm HP}(\Delta)$, called the Hodge polygon, is given by Adolphson-Sperber \cite{AS2} in terms of  lattice points 
in a ``fundamental domain'' of the convex polytope $\Delta$ defined using the exponents of monomials in $G(t_0,\ldots, t_n)$. 
To get a feeling what the Hodge polygon looks like, see the end of this paper for an explicit closed formula 
in the case $A=B=1$. 
We are interested in deciding when the generic Newton polygon coincides with its lower bound, 
i.e., when the universal family of $(A,B)$-polynomials is generically ordinary for its $L$-function. 
Unlike the universal family of Deligne polynomials of degree $d$, the polytope for the universal family of $(A,B)$-polynomials is no longer a simplex.  
There is no elementary diagonal example available and the problem becomes deeper. We apply the facial decomposition theorem in \cite{W1} to 
prove the following result. 

\begin{theorem} \label{IV} For fixed positive integers $d, A, B$ relatively prime to $p$, and a 
non-negative integer $e\leq dA/(A+B)$,  the universal family of $(A,B)$-polynomial  
$G(t_0,..., t_n)$ defined in (\ref{G}) with $\deg(g)\leq e$ is generically ordinary for its $L$-function if $p\equiv 1 \mod[ A, dB]$, where  $[A, dB]$ denotes 
the least common multiple of $A$ and $dB$. 
\end{theorem}

When $e>dA/(A+B)$, the polytope of the corresponding universal $(A,B)$-polynomials 
is more complicated, having three (instead of two) codimension $1$ faces not containing the 
origin, and one of them is not a simplex. The decomposition theorems in \cite{W1} can still be used to obtain similar results. 
The question is how to cleverly apply the various decomposition theorems to obtain 
an optimal (smallest) modulus $D$ for the arithmetic progression $p\equiv 1 \mod D$. 
We leave this to interested readers.  

The paper is organized as follows. In Section 1, we prove Theorem \ref{I} using Denef-Loeser's results under the assumption that 
the Deligne polynomial $f(t_0, \ldots, t_n)$ is actually affine Dwork regular. 
In this case, the related cohomology group is pure. In Section 2, we deduce Theorems \ref{I} and \ref{II} 
from the results in Section 1 using a specialization argument. In both sections, we actually work with 
the $(A, B)$-exponential sums twisted by a multiplicative character. In Section 3, we prove Theorem \ref{III}
about the improvement of the constant in the bound. In Section 4, we study the generic ordinary property 
for the $(A, B)$-polynomial and prove Theorem \ref{IV}. 

%If $e\geq d$, the $(A,B)$-polynomial $G(t)$ is non-longer 
%a family of Deligne polynomials. To get good estimates, the smooth condition on the leading form $f_d$ of $f$ 
%is not enough. One also needs smooth conditions on the leading form $g_e$ of $g$ and the intersection 
%$\{f_d=g_e=0\}$. 

\section{Non-degenerate and pure case}

Let $k$ be a finite field of characteristic $p$ with $q$ elements and let 
$$f=\sum_{j=1}^N a_j t_1^{w_{1j}}\cdots t_n^{w_{nj}}\in k[t_1^{\pm 1}, \ldots, t_n^{\pm 1}]$$
be a Laurent polynomial. 
Define the Newton polytope $\Delta_\infty(f)$ of $f$ at $\infty$ to be 
the convex hull of $\{0, \mathbf w_1, \ldots, \mathbf w_N\}$ in $\mathbb R^n$, where 
$\mathbf w_j=(w_{1j},\ldots, w_{nj})\in\mathbb Z^n$. We say $f$ is \emph{nondegenerate with respect to} 
$\Delta_\infty(f)$ if for any face $\tau$ of $\Delta_\infty(f)$ not containing the origin, the subscheme of
$\mathbb G_m^n$ defined by 
$$\frac{\partial f_\tau}{\partial t_1}=\cdots=\frac{\partial f_\tau}{\partial t_n}=0$$ is empty, where 
$$f_\tau=\sum_{\mathbf w_j\in \tau}a_j t_1^{w_{1j}}\cdots t_n^{w_{nj}}.$$ 

Suppose $f\in k[t_1^{\pm 1}, \ldots,t_r^{\pm 1}, t_{r+1},\ldots, t_n]$ is a polynomial with respect to the coordinates 
$t_{r+1}, \ldots, t_n$. We say $f$ is \emph{commode with respect to the coordinates} $t_{r+1}, \ldots, t_n$ 
if for any subset $S\subset \{r+1, \ldots, n\}$, we have
$$\mathrm{dim}\Big( \Delta_\infty(f)\cap  \{(w_1, \ldots, w_n)\in\mathbb R^n:\, w_j=0 \hbox { for all } j\in S\}\Big)= n-\# S.$$ 

Suppose $f\in k[t_1, \ldots, t_n]$ is a polynomial of degree $d$. Recall that $f$ is a \emph{Deligne polynomial} if 
$d$ is relatively prime to $p$, and the homogeneous degree $d$ part $f_d$ of $f$ defines a smooth hypersurface $f_d=0$ in 
$\mathbb P^{n-1}$. A homogeneous polynomial $F(t_1, \ldots, t_n)$ is called \emph{Dwork-regular with respect to 
the coordinates} $t_1, \ldots, t_n$ if the subscheme of $\mathbb P^{n-1}$ defined by 
$$F=t_1\frac{\partial F}{\partial t_1}=\cdots=t_n\frac{\partial F}{\partial t_n}=0$$ is empty. If 
the degree $d$ of $F$ is prime to $p$, we may omit the condition $F=0$ since 
$$dF=\sum_i t_i\frac{\partial F}{\partial t_i}.$$
Again under the assumption that $(d, p)=1$, 
$F$ is Dwork-regular if and only if for any nonempty subset $S\subset \{1, \ldots, n\}$, the homogeneous polynomial 
$F_S$ obtained from $F$ by setting $t_i=0$ $(i\not\in S)$ is a nonzero Deligne polynomial in the variables $t_i$ $(i\in S)$.
Dwork shows that if $F$ is a homogenous Deligne polynomial, then there exists a finite extension 
$k'$ of $k$ and a matrix $(a_{ij})\in \mathrm{GL}(n, k')$ such that $F$ is Dwork-regular with respect to the coordinates
$Y_i=\sum_j a_{ij} X_j$. Confer \cite[Lemma 3.1]{K}. 

\begin{proposition} Let $f$ be a polynomial of degree $d$ prime to $p$, and let $f_d$ be the homogeneous degree $d$ part of $f$. 
If $f_d$ is Dwork regular, then $f$ is commode, its Newton polytope $\Delta_\infty(f)$ at $\infty$  is the simplex with vertices 
$$(0, \ldots, 0), (d, 0, \ldots, 0), \ldots, (0, \ldots, 0, d),$$ and $f$
is nondegenerate with respect to $\Delta_\infty(f)$.
\end{proposition}

\begin{proof} Suppose $f_d$ is Dwork regular. Then for each $j\in \{1, \ldots, n\}$, the coefficient of 
$t_j^d$ in $f$ is nonzero. Otherwise, let $P_j$ be the point in $\mathbb P^{n-1}$ whose only nonzero homogenous coordinate is the $j$-th.  
Then $P_j$ is a point in $$f_d=t_1\frac{\partial f_d}{\partial t_1}=\cdots=t_n\frac{\partial f_d}{\partial t_n}=0.$$ So $f$ is commode and its Newton polytope 
$\Delta_\infty(f)$ at $\infty$ is the simplex with vertices 
$$(0, \ldots, 0), (d, 0, \ldots, 0), (0, d, 0,\ldots, 0), \ldots, (0, ,\ldots, 0,d).$$ 
From the condition that $f_d$ is Dwork-regular, one deduces that is  $f$ is nondegenerate with respect to $\Delta_\infty(f)$. 
\end{proof}

A polynomial $f$ of degree $d$ is called \emph{affine-Dwork-regular with respect to the coordinates} $t_1,\ldots, t_n$ 
if the homogenization 
$$F(t_0, t_1, \ldots t_n)=t_0^d f\Big(\frac{t_1}{t_0},\ldots, \frac{t_n}{t_0}\Big)$$ 
is Dwork-regular with respect to $t_0, t_1, \ldots, t_n$. This condition implies that the degree $d$ part $f_d$ of $f$ 
is Dwork-regular since $f_d=F(0, t_1, \ldots, t_n)$. In \cite[Lemma 3.2]{K}, Katz shows that for any Deligne polynomial $f$ 
so that its leading form $f_d$ is Dwork-regular with respect to $t_1, \ldots, t_n$, the polynomial $f(x)+a$ is affine-Dwork-regular
for all but finitely many $a\in \overline k$. 

Suppose $f(t_1, \ldots, t_n)$ is a polynomial of degree $d$ prime to $p$ and $g(t_1, \ldots, t_n)$ a polynomial of degree $<d$. Let 
$A$ and $B$ be positive integers and let $P_B(s)$ be a one-variable polynomial of degree $\leq B$. Let 
$G: \mathbb G_m\times\mathbb A^n\to \mathbb A^1$ be the morphism 
$$(t_0, t_1, \ldots, t_n)\mapsto G(t_0,\ldots, t_n)=t_0^A f(t_1, \ldots, t_n) +g(t_1, \ldots, t_n)+P_B(1/t_0).$$
Choose a prime number $\ell$ distinct from $p$ and fix a nontrivial additive character $$\psi: (k,+)\to\overline{\mathbb Q}_\ell^*.$$ Denote by $\mathcal L_\psi$
the Artin-Schreier sheaf on $\mathbb A^1$ corresponding to $\psi$. Let $$\chi:(k^*,\times) \to\overline{\mathbb Q}_\ell$$ be a multiplicative character, let 
$\mathcal K_\chi$ be the associated Kummer sheaf on $\mathbb G_m$, and let $$\pi_1: \mathbb G_m\times\mathbb A^n\to \mathbb G_m$$ be the projection.

\begin{theorem}\label{dl} Suppose that $f$ is an affine-Dwork-regular polynomial of degree $d$ prime to $p$, 
that $p$ is prime to $AB$,  that $P_B$ is of degree $B$, and $\mathrm{deg}(g)<\frac{Bd}{A+B}$. 

(i) We have $H^i_c(\mathbb G_{m, \bar k}\times \mathbb A_{\bar k}^n, \pi_1^*\mathcal K_\chi\otimes G^*\mathcal L_\psi)=0$ for $i\not =n+1$, 
and $$\mathrm{dim}\, H^{n+1}_c(\mathbb G_{m, \bar k}\times \mathbb A_{\bar k}^n,  \pi_1^*\mathcal K_\chi\otimes G^*\mathcal L_\psi)=(A+B)(d-1)^n.$$ 

(ii)  $H^{n+1}_c(\mathbb G_{m, \bar k}\times \mathbb A_{\bar k}^n, \pi_1^*\mathcal K_\chi\otimes G^*\mathcal L_\psi)$ is pure of weight $n+1$. 

(iii) We have $$\Big \vert \sum_{t_0\in k^*}\sum_{t_1,\ldots, t_n\in k} \chi(t_0)\psi\Big(t_0^A f(t_1, \ldots, t_n) +g(t_1, \ldots, t_n)+P_B(1/t_0)\Big)\Big\vert
\leq (A+B)(d-1)^n q^{(n+1)/2},$$ where $\vert\cdot \vert$ is the composite of an arbitrary isomorphism $\overline{\mathbb Q}_\ell
\stackrel\cong\to \mathbb C$ and the absolute value on $\mathbb C$. 
\end{theorem}

\begin{proof} We first treat the case where $\chi=1$ is trivial. 
Since $f$ is affine-Dwork-regular, $\mathrm{deg}\, P_B=B$, and $\mathrm{deg}(g)<\frac{Bd}{A+B}$, 
the Laurent polynomial $t_0^A f(t_1, \ldots, t_n) +g(t_1, \ldots, t_n)+P_B(1/t_0)$ is commode in $t_1, \ldots, t_n$, 
its Newton polytope at $\infty$ is the simplex $\Delta$ in $\mathbb R^{n+1}$ with vertices 
$$(-B, 0, \ldots, 0), (A, 0, \ldots, 0), (A, d, 0,\ldots, 0), (A, 0, d,0,\ldots, 0), \ldots, (A, 0, \ldots, 0, d),$$
and it is nondegenerate with respect respect to $\Delta$. 
Here because $\mathrm{deg}(g)<\frac{Bd}{A+B}$, the exponents of $g$ lies in the interior of $\Delta$. 
For any subset $S\subset\{1, \ldots, n\}$. We have 
\begin{eqnarray*}
&&\mathrm{vol}\Big(\Delta\cap\{(t_0, t_1,\ldots, t_n)\in\mathbb R^{n+1}: \,t_i=0\hbox{ for all } i\in S\}\Big)=\frac{1}{(n+1-\# S)!}d^{n-\#S} (A+B),\\
&&\sum_{S\subset \{1, \ldots, n\}} (-1)^{\# S} d^{n-\#S}(A+B)=(A+B)(d-1)^n.
\end{eqnarray*}
The assertions (i)-(ii) follows directly from \cite[Theorem 9.2]{DL}. 

Next we consider the general $\chi$ case. Let $m=q-1$, and let $[t_0^m]$ be the Galois \'etale covering  
$$[t_0^m]: \mathbb G_m\to \mathbb G_m,\quad t_0\mapsto t_0^m.$$ We have 
$$\bigoplus_{\chi:k^*\to\overline{\mathbb Q}_\ell}\mathcal K_\chi=[t_0^m]_*\overline{\mathbb Q}_\ell.$$
By the proper base change theorem, the projection formula and the Leray spectral sequence, we have 
\begin{eqnarray*}
&&\bigoplus_{\chi: k^*\to\overline{\mathbb Q}_\ell} H_c^{i}(\mathbb G_{m,\bar k}\times\mathbb A^n_{\bar k}, \pi_1^*\mathcal K_\chi\otimes G^*\mathcal L_\psi)\\
&\cong& H_c^{i}(\mathbb G_{m,\bar k}\times\mathbb A^n_{\bar k},  \pi_1^* [t_0^m]_*\overline{\mathbb Q}_\ell\otimes G^*\mathcal L_\psi)\\
&\cong& H_c^{i}(\mathbb G_{m,\bar k}\times\mathbb A^n_{\bar k}, ([t_0^m]\times\mathrm{id}_{\mathbb A^n})_* 
([t_0^m]\times\mathrm{id}_{\mathbb A^n})^* G^*\mathcal L_\psi)\\
&\cong& H_c^{i}(\mathbb G_{m,\bar k}\times\mathbb A^n_{\bar k}, ([t_0^m]\times\mathrm{id}_{\mathbb A^n})^* G^*\mathcal L_\psi)\\
&\cong& H_c^{i}(\mathbb G_{m,\bar k}\times\mathbb A^n_{\bar k}, H^*\mathcal L_\psi),
\end{eqnarray*}
where $H$ is the morphism
$$H:\mathbb G_m\times\mathbb A^n\to \mathbb A^1, \quad (t_0, \ldots, t_n)\mapsto H(t_0, \ldots, t_n)=G(t_0^m, t_1, \ldots, t_n).$$
Applying the $\chi=1$ case to the polynomial $H(t_0, \ldots, t_n)$ (and with $A$ and $B$ replaced by $mA$ and $mB$, respectively), we see 
that $H_c^{i}(\mathbb G_{m,\bar k}\times\mathbb A^n_{\bar k}, H^*\mathcal L_\psi)=0$ for $i\not=n+1$, 
$$\mathrm{dim}\, H_c^{n+1}(\mathbb G_{m,\bar k}\times\mathbb A^n_{\bar k}, H^*\mathcal L_\psi)=m(A+B)(d-1)^n,$$
and 
$H_c^{n+1}(\mathbb G_{m,\bar k}\times\mathbb A^n_{\bar k},H^*\mathcal L_\psi)$ is pure of weight $n+1$. 
Since $\mathcal K_\chi$ is tamely ramified, we have 
$$\chi_c(\mathbb G_{m,\bar k}\times\mathbb A^n_{\bar k},  \pi_1^*\mathcal K_\chi\otimes G^*\mathcal L_\psi)
=\chi_c(\mathbb G_{m,\bar k}\times\mathbb A^n_{\bar k},  G^*\mathcal L_\psi)$$ for all $\chi$ by \cite[2.1]{I}. 
The assertions (i) and (ii) follow immediately. 
The assertion (iii) follows then from the Grothendieck trace formula
\begin{eqnarray*}
&&\sum_{t_0\in k^*}\sum_{t_1,\ldots, t_n\in k} \chi(t_0)\psi\Big(t_0^A f(t_1, \ldots, t_n) +g(t_1, \ldots, t_n)+P_B(1/t_0)\Big)\\
&=&\sum_i (-1)^i
\mathrm{Tr}\Big(\mathrm{Fr}, H^i_c(\mathbb A_{\bar k}^n \times \mathbb G_{m, \bar k},  \pi_1^*\mathcal K_\chi\otimes G^*\mathcal L_\psi)\Big).
\end{eqnarray*}
\end{proof}

The following proposition is due to Deligne. 

\begin{proposition}\label{deligne} Notation as above. Suppose $f(t_1, \ldots, t_n)$ is a Deligne polynomial of degree $d$
and suppose $g(t_1, \ldots, t_n)$ is of degree $<d$. Let $\pi_1:\mathbb G_m\times \mathbb A^n\to \mathbb G_m$ be the projection.
Then $R^i\pi_{1!} G^* \mathcal L_\psi$ vanishes for $i\not=n$, and the sheaf
$\mathcal F=R^n\pi_{1!} G^* \mathcal L_\psi$ is a lisse sheaf on $\mathbb G_m$ with rank $(d-1)^n$ and pure of weight $n$. We have
$$H_c^i(\mathbb G_{m,\bar k}, \mathcal F\otimes 
\mathcal K_\chi)\cong H_c^{i+n}(\mathbb G_{m, \bar k}\times\mathbb A^n_{\bar k}, \pi_1^*\mathcal K_\chi\otimes G^*\mathcal L_\psi).$$
\end{proposition} 

\begin{proof} The first two assertions follow directly form  \cite[3.7.3 and 3.7.2.3]{D}. Note that since we assume $\mathrm{deg}(g)<d$, 
for each parameter $t_0$, 
$t_0^A f(t_1, \ldots, t_n) +g(t_1, \ldots, t_n)+P_B(1/t_0)$ is a Deligne polynomials in the variables $t_1, \ldots, t_n$. 
By the projection formula, we have 
\begin{eqnarray*}
R\Gamma_c(\mathbb G_{m,\bar k}\times\mathbb A^n_{\bar k},
\pi_1^*\mathcal K_\chi\otimes G^*\mathcal L_\psi)
&\cong& R\Gamma_c\Big(\mathbb G_{m,\bar k}, 
R\pi_{1!} (\pi_1^*\mathcal K_\chi\otimes G^*\mathcal L_\psi)\Big)\\
&\cong& R\Gamma_c(\mathbb G_{m,\bar k},  \mathcal K_\chi\otimes
R\pi_{1!} G^*\mathcal L_\psi).
\end{eqnarray*}
So we have a spectral sequence 
$$ H^i_c(\mathbb G_{m,\bar k}, \mathcal K_\chi\otimes R^j\pi_{1!} G^*\mathcal L_\psi)\Rightarrow H_c^{i+j}(\mathbb G_{m,\bar k}\times\mathbb A^n_{\bar k},
\pi_1^*\mathcal K_\chi\otimes G^*\mathcal L_\psi),$$ and it degenerates by the above results of Deligne. So we have
$$H_c^i(\mathbb G_{m,\bar k},\mathcal K_\chi\otimes \mathcal F)\cong H_c^{i+n}(\mathbb G_{m,\bar k}\times \mathbb A^n_{\bar k},
\pi_1^*\mathcal K_\chi\otimes G^*\mathcal L_\psi).$$
\end{proof}

\begin{remark} Deligne's result \cite[3.7.2.3]{D} also follows from \cite[Theorem 9.2]{DL}. To see this, we may replace $k$ by 
any finite extension, or replace the coordinates $t_1, \ldots, t_n$ by any linear change. 
So by \cite[Lemmas 3.1]{K}, we may assume the leading form $f_d$ of $f$ is Dwork-regular. 
Then $f$ is commode, the Newton polytope $\Delta_\infty(f)$ of $f$ at $\infty$ is the simplex with vertices
$$(0, \ldots, 0), (d, 0,\ldots,0), \ldots, (0, \ldots, 0,d),$$ and $f$ is
nondegenerate with
respect to $\Delta_\infty(f)$. For any subset $S\subset\{1, \ldots, n\}$, we have 
\begin{eqnarray*}
&& \mathrm{vol}\Big(\Delta_\infty(f)\cap\{(t_1,\ldots, t_n)\in\mathbb R^{n}: \,t_i=0\hbox{ for all } i\in S\}\Big)
=\frac{1}{(n-\# S)!}d^{n-\#S},\\
&& \sum_{S\subset \{1, \ldots, n\}} (-1)^{\# S} d^{n-\#S}=(d-1)^n.
\end{eqnarray*}
So \cite[3.7.2.3]{D} follows from \cite[Theorem 9.2]{DL}.
\end{remark}

As a direct consequence of Theorem \ref{dl} and Proposition \ref{deligne}, we have the following. 

\begin{theorem} [\cite{K} Theorems 5.1 and 8.1]  Suppose that $f$ is an affine-Dwork-regular polynomial of degree $d$ prime to $p$, 
that $p$ is prime to $AB$,  that $P_B$ is of degree $B$, and $\mathrm{deg}(g)<\frac{Bd}{A+B}$. Let 
$\mathcal F=R^n\pi_{1!} G^* \mathcal L_\psi.$

(i) We have $H^i_c(\mathbb G_{m, \bar k}, \mathcal F\otimes \mathcal K_\chi)=0$ for $i\not =1$,
and $$\mathrm{dim}\, H^1_c(\mathbb G_{m, \bar k}, \mathcal F\otimes\mathcal K_\chi)=(A+B)(d-1)^n.$$ 

(ii)  $H^1_c(\mathbb G_{m,\bar k}, \mathcal F\otimes\mathcal K_\chi)$ is pure of weight $n+1$. 
\end{theorem}

\begin{remark} There is no need to assume $p$ is prime to $A+B$ in \cite[Theorems 5.1 and 8.1, Corollary 8.2]{K}, and no 
need to assume $p$ is odd in \cite[Theorems 1.1, 2.1, 3.5]{K}.  
\end{remark}

\section{Degenerate and mixed case}

\begin{lemma}\label{translate} Notation as above. For any $k$-point $a, b\in\mathbb A^1(k)$, let 
$[at^A]$ and $[bt^{-B}]$ be the morphisms 
\begin{eqnarray*}
\mathbb G_m\to \mathbb A^1, && t\mapsto a t^A,\\
\mathbb G_m\to \mathbb A^1, && t\mapsto b t^{-B},
\end{eqnarray*}
respectively, and let $$\mathcal F_\chi=\mathcal F\otimes\mathcal K_\chi=R^n\pi_{1!}G^*\mathcal L_\psi\otimes \mathcal K_\chi.$$ 
In the triangulated category $D_c^b(\mathbb G_m, \overline{\mathbb Q}_\ell)$ of complexes of $\overline{\mathbb Q}_\ell$-sheaves
on $\mathbb G_m$, we have 
\begin{eqnarray*}
\mathcal F_\chi\otimes [at^A]^*\mathcal L_\psi &\cong& R\pi_{1!} (\pi_1^*\mathcal K_\chi\otimes G_a^* \mathcal L_\psi) [n],\\
\mathcal F_\chi\otimes [bt^{-B}]^*\mathcal L_\psi &\cong& R\pi_{1!} (\pi_1^*\mathcal K_\chi\otimes G_b^* \mathcal L_\psi) [n],
\end{eqnarray*}
where
$G_a, G_b:\mathbb G_m\times \mathbb A^n\to \mathbb A^1$ are the morphisms 
\begin{eqnarray*}
(t_0, t_1, \ldots, t_n)&\mapsto& t_0^A \Big(f(t_1, \ldots, t_n) +a\Big)+g(t_1, \ldots, t_n)+P_B(1/t_0),\\
(t_0, t_1, \ldots, t_n)&\mapsto& t_0^A f(t_1, \ldots, t_n)+g(t_1, \ldots, t_n)+\Big(P_B(1/t_0)+ bt_0^{-B}\Big)
\end{eqnarray*}
\end{lemma}

\begin{proof} By Proposition \ref{deligne}, we have 
$$\mathcal F= R\pi_{1!} G^* \mathcal L_\psi [n].$$
By the projection formula, we have 
\begin{eqnarray*}
\mathcal F\otimes \mathcal K_\chi \otimes [at^A]^*\mathcal L_\psi&\cong& 
R\pi_{1!} G^* \mathcal L_\psi [n]\otimes\mathcal K_\chi\otimes[at^A]^* \mathcal L_\psi[n]\\
&\cong&R\pi_{1!}\Big (\pi_1^*\mathcal K_\chi\otimes G^* \mathcal L_\psi \otimes \pi_1^*[at^A]^*\mathcal L_\psi\Big)[n].
\end{eqnarray*}
We then use the fact that 
$$G^* \mathcal L_\psi \otimes \pi_1^*[at^A]^*\mathcal L_\psi\cong G_a^*\mathcal L_\psi$$ which follows from 
\cite[Sommes trig. 1.7.1] {SGA4.5}. Similarly we can prove the second isomorphism. 
\end{proof} 

\begin{lemma}\label{pure} Let $X$ be a geometrically connected smooth projective curve over $k$, $S$ a finite
closed subscheme of $X$,  $j: X-S\hookrightarrow X$ the canonical open immersion, and $\mathcal F$ a
pure lisse $\overline{\mathbb Q}_\ell$-sheaf on $X-S$ of weight $w$. Suppose $H_c^1((X-S)_{\bar k}, \mathcal F)$ is pure of weight $w+1$
and $S(\bar k)$ contains at least two points. Then we have $j_!\mathcal F\cong j_*\mathcal F$. 
\end{lemma}

\begin{proof} From the short exact sequence 
$$0\to j_!\mathcal F\to j_*\mathcal F\to j_*\mathcal F/j_!\mathcal F\to 0,$$ we get the exact sequence 
$$\begin{array}{ccccccc}
0\to&H^0_c((X-S)_{\bar k}, \mathcal F)&\to H^0(X_{\bar k}, j_*\mathcal F)&\to& H^0(X_{\bar k},  j_*\mathcal F/j_!\mathcal F)&\to &H^1_c((X-S)_{\bar k}, \mathcal F).\\
&\parallel&&&\wr\!\!\parallel&&\\
&0&&&\bigoplus_{\bar x\in S(\bar k)}(j_*\mathcal F)_{\bar x}&&
\end{array}$$
By \cite[1.8.1.]{D}, the weights of $(j_*\mathcal F)_{\bar x}$ are $\leq w$. Since $H_c^1((X-S)_{\bar k}, \mathcal F)$
is pure of weight $w+1$, the last arrow vanishes. So we get an isomorphism 
$$H^0(X_{\bar k},j_* \mathcal F)\stackrel\cong \to \bigoplus_{\bar x\in S(\bar k)}(j_*\mathcal F)_{\bar x}.$$
Since $\mathcal F$ is a lisse sheaf, each  canonical homomorphism 
$$H^0(X_{\bar k}, j_*\mathcal F)\stackrel\cong \to (j_*\mathcal F)_{\bar x}$$ is injective. Thus each projection 
$$\bigoplus_{\bar x\in S(\bar k)}(j_*\mathcal F)_{\bar x}\to (j_*\mathcal F)_{\bar x}$$
is injective.  Since $S(\bar k)$ contains at least two points, we have $(j_*\mathcal F)_{\bar x}=0$ for all $x\in S(\bar k)$. 
So $j_!\mathcal F\cong j_*\mathcal F$. 
\end{proof}

The following corollary proves Theorems \ref{I} and \ref{II} by taking $\chi=1$. 

\begin{corollary}\label{K} [\cite{K} Corollaries 5.2 and 8.2]  Suppose $d$ is prime to $p$, $p$ is prime to $AB$, and 
$\mathrm{deg}(g)<\frac{Bd}{A+B}$. Suppose furthermore
one of the following conditions holds:

(a) $f$ is a Deligne polynomial of degree $d$, and $\mathrm{deg}(P_B)=B$. 

(b) $f$ is affine-Dwork-regular and $\mathrm{deg}\, P_B\leq B$. 

\noindent Let  $\mathcal F=R^n\pi_{1!} G^* \mathcal L_\psi$.

(i) We have 
$H^i_c(\mathbb G_{m, \bar k}, \mathcal F\otimes\mathcal K_\chi)=0$ for $i\not =1$, and 
$$\mathrm{dim}\,H^1_c(\mathbb G_{m, \bar k}, \mathcal F\otimes\mathcal K_\chi) 
\leq (A+B)(d-1)^n.$$ 

(ii) We have $H^i_c(\mathbb G_{m, \bar k}\times \mathbb A_{\bar k}^n, \pi_1^*\mathcal K_\chi\otimes G^*\mathcal L_\psi)=0$ for 
$i\not=n+1$ and $$\mathrm{dim}\, H^{n+1}_c(\mathbb G_{m, \bar k}\times \mathbb A_{\bar k}^n, \pi_1^*\mathcal K_\chi\otimes G^*\mathcal L_\psi)
\leq (A+B)(d-1)^n.$$

(iii) We have $$\Big \vert \sum_{t_0\in k^*}\sum_{t_1,\ldots, t_n\in k} \chi(t_0) \psi\big(t_0^A f(t_1, \ldots, t_n) +g(t_1, \ldots, t_n)+P_B(1/t_0)\big)\Big\vert
\leq (A+B)(d-1)^n q^{(n+1)/2}.$$ 
\end{corollary} 

\begin{proof} (i) We first work under the condition (a). Let $\mathrm{pr}_i: \mathbb A^1\times \mathbb A^1\to \mathbb A^1$ $(i=1,2)$
be the projections and let  
$$\langle \,,\,\rangle: \mathbb A^1\times \mathbb A^1\to \mathbb A^1,\quad (t, t')\mapsto tt'$$ be the canonical pairing
on $\mathbb A^1$. Recall that the Deligne-Fourier transform is the functor $$\mathfrak F: D_c^b(\mathbb A^1, \overline{\mathbb Q}_\ell)
\to D_c^b(\mathbb A^1, \overline{\mathbb Q}_\ell),\quad \mathfrak F(K)=R\mathrm{pr}_{2!}\Big(\mathrm{pr}_1^*K\otimes\langle\,,\,\rangle^*\mathcal L_\psi\Big)[1].$$
Let $[t^A]: \mathbb G_m\to \mathbb G_m$ be the finite \'etale morphism 
$t\mapsto t^A$ and let $j:\mathbb G_m \hookrightarrow \mathbb A^1$ be the canonical open immersion. 
Fix notation by the following diagram:
$$\begin{array}{ccccccc}
\mathbb G_m\times\mathbb A^1&\stackrel{[t^A]\times\mathrm{id}}\to&\mathbb G_m\times\mathbb A^1&
\stackrel{j\times \mathrm{id}}\hookrightarrow&
\mathbb A^1\times\mathbb A^1&\stackrel{\mathrm{pr}_2}\to&\mathbb A^1\\
{\scriptstyle q_1}\downarrow&&\downarrow&&{\scriptstyle\mathrm{pr}_1}\downarrow&&\downarrow\\
\mathbb G_m&\stackrel{[t^A]}\to&\mathbb G_m&\stackrel j\hookrightarrow&
\mathbb A^1&\to&\mathrm{Spec}\, k.
\end{array}$$
Again let $\mathcal F_\chi=\mathcal F\otimes\mathcal K_\chi$. By the proper base change theorem and the projection formula, we have
\begin{eqnarray*}
\mathfrak F(j_![t^A]_!\mathcal F_\chi[1]) &\cong& R\mathrm{pr}_{2!}\Big(\mathrm{pr}_1^* (j\circ [t^A])_! \mathcal F_\chi
\otimes \langle\, ,\,\rangle^* \mathcal L_\psi\Big)[2]\\
&\cong&R\mathrm{pr}_{2!}\Big(\big((j\circ [t^A])\times\mathrm{id}\big)_! q_1^* \mathcal F_\chi \otimes \langle\, ,\,\rangle^* \mathcal L_\psi\Big)[2]\\
&\cong& R\mathrm{pr}_{2!} \big((j\circ [t^A])\times\mathrm{id}\big)_!\Big( q_1^*\mathcal F_\chi\otimes \big((j\circ[t^A])\times\mathrm{id}\big)^* 
\langle\, ,\,\rangle^*\mathcal L_\psi\Big)[2]\\
&\cong& Rq_{2!}\Big(q_1^*\mathcal F_\chi\otimes[t^At']^*\mathcal L_\psi\Big)[2],
\end{eqnarray*}
where $q_2: \mathbb G_m\times\mathbb A^1\to \mathbb A^1$ is the projection, and 
$[t^At']$ is the morphism
$$\mathbb G_m\times\mathbb A^1\to \mathbb A^1, \quad (t, t')\mapsto t^A t'.$$ 
So for any geometric point $a\in \mathbb A^1(\bar k)$, we have 
\begin{eqnarray}\label{fourier}
\mathfrak F(j_![t^A]_!\mathcal F_\chi[1])_{\bar a}
\cong R\Gamma(\mathbb G_{m,\bar k} , \mathcal F_\chi\otimes [at^A]^*\mathcal L_\psi)[2]
\end{eqnarray}
Combined with Lemma \ref{translate}, we get
\begin{eqnarray}\label{fact1}
\begin{array}{rl}
\mathfrak F(j_![t^A]_!\mathcal F_\chi[1])_{\bar a}
\cong &R\Gamma\Big(\mathbb G_{m,\bar k}, R\pi_{1!} (\pi_1^*\mathcal K_\chi\otimes G_a^*\mathcal L_\psi)\Big)[n+2]\\
\cong &R\Gamma(\mathbb G_{m, \bar k}\times \mathbb A^n_{\bar k}, \pi_1^*\mathcal K_\chi\otimes G_a^*\mathcal L_\psi)[n+2].
\end{array}
\end{eqnarray}
Let $U$ be the Zariski open set of $\mathbb A^1$ so that for any $a\in U(\bar k)$, 
$f(t_1,\ldots, t_n)+a$ is affine-Dwork-regular. Then by Theorem \ref{dl} (i), for any $a\in U(\bar k)$, we have 
\begin{eqnarray*}
\mathfrak F(j_![t^A]_!\mathcal F_\chi[1])_{\bar a}&\cong& H_c^{n+1}( \mathbb G_{m, \bar k}\times\mathbb A_{\bar k}^n, 
\pi_1^*\mathcal K_\chi\otimes G_a^*\mathcal L_\psi)[1],\\
\mathrm{dim}\, \mathfrak F(j_![t^A]_!\mathcal F_\chi[1])_{\bar a}&=& -(A+B)(d-1)^n.
\end{eqnarray*}
The sheaf $j_![A]_!\mathcal F_\chi[1]$ is
a perverse sheaf. By \cite[1.3.2.3]{L},  $\mathfrak F(j_![t^A]_!\mathcal F_\chi[1])$ is perverse, and its generic rank 
is $-(A+B)(d-1)^n$. As a perverse sheaf, for any $s\in \mathbb A^1(\overline k)$, the specialization homomorphism
 $$H^{-1}(\mathfrak F(j_![t^A]_!\mathcal F_\chi[1])_{\bar s})\to H^{-1}(\mathfrak F(j_![t^A]_!\mathcal F_\chi[1])_{\bar\eta})$$
 is injective, where $\bar\eta$ is the geometric generic point of $\mathbb A^1$. It follows that 
 \begin{eqnarray}\label{inequality}
 \mathrm{dim}\, H^{-1}(\mathfrak F(j_![t^A]_!\mathcal F_\chi[1])_{\bar s})\leq (A+B) (d-1)^n.
 \end{eqnarray}
 By (\ref{fourier}), we have 
 $$H^{-1}(\mathfrak F(j_![t^A]_!\mathcal F_\chi[1])_{\bar s})\cong 
 H_c^1(\mathbb G_{m,\bar k} , \mathcal F_\chi\otimes [st^A]^*\mathcal L_\psi).$$
 Taking $s=0$, we have 
 $$H^{-1}(\mathfrak F(j_![t^A]_!\mathcal F_\chi[1])_{\bar 0})\cong H^1_c(\mathbb G_{m, \bar k}, \mathcal F_\chi).$$ So 
by (\ref{inequality}), we have
 $$\mathrm{dim}\,H^1_c(\mathbb G_{m, \bar k}, \mathcal F_\chi) 
\leq (A+B)(d-1)^n.$$
By (\ref{fourier}) and (\ref{fact1}), for any $a\in U(\bar k)$, we have
$$H^{-1}(\mathfrak F(j_![A]_!\mathcal F_\chi[1])_{\bar a})
\cong H_c^1(\mathbb G_{m, \bar k}, \mathcal F_\chi\otimes [at^A]^*\mathcal L_\psi)
\cong H_c^{n+1}(\mathbb G_{m, \bar k}\times \mathbb A_{\bar k}^n, \pi_1^*\mathcal K_\chi\otimes G_a^*\mathcal L_\psi).$$
So $H_c^1(\mathbb G_{m,\bar k} , \mathcal F_\chi\otimes [at^A]^*\mathcal L_\psi)$ is pure of weight $n+1$ by Theorem \ref{dl} (ii). 
Let $\bar j:\mathbb G_m\hookrightarrow
\mathbb P^1$ be the canonical open immersion. 
By Lemma \ref{pure}, we have 
$$\bar j_!(\mathcal F_\chi\otimes [at^A]^*\mathcal L_\psi)\cong\bar j_*(\mathcal F_\chi\otimes [at^A]^*\mathcal L_\psi).$$
Since $[at^A]^*\mathcal L_\psi$ is lisse on $\mathbb A^1$, this implies that 
$$j_!\mathcal F_\chi\cong j_*\mathcal F_\chi.$$ We have $H^0(\mathbb A^1_{\bar k}, j_!\mathcal F_\chi)=0$ since $\mathcal F_\chi$ is lisse on 
$\mathbb G_m$. 
So  $$H^0(\mathbb G_{m,\bar k}, \mathcal F_\chi)\cong 
H^0(\mathbb A^1_{\bar k}, j_*\mathcal F_\chi)\cong H^0(\mathbb A^1_{\bar k}, j_!\mathcal F_\chi)=0.$$ Then by the Poincar\'e duality, we have
$H^2_c(\mathbb G_{m,\bar k},\mathcal F_\chi)=0$. Finally $H^0_c(\mathbb G_{m,\bar k},\mathcal F_\chi)=0$ since $\mathbb G_m$ is affine. 

Next we work under the condition (b). For any geometric point $b\in \mathbb A^1(\bar k)$, we have 
\begin{eqnarray}\label{fact2}
\begin{array}{rl}
\mathfrak F(j_![t^{-B}]_!\mathcal F_\chi[1])_{\bar b}
\cong  &R\Gamma(\mathbb G_{m,\bar k} , \mathcal F_\chi\otimes [bt^{-B}]^*\mathcal L_\psi)[2]\\
\cong &R\Gamma(\mathbb G_{m, \bar k}\times \mathbb A^n_{\bar k},  \pi_1^*\mathcal K_\chi\otimes G_b^*\mathcal L_\psi)[n+2].
\end{array}
\end{eqnarray}
By Theorem \ref{dl}, for any $b\in \mathbb G_m(\bar k)$, we have 
\begin{eqnarray*}
\mathfrak F(j_![t^{-B}]_!\mathcal F_\chi[1])_{\bar b}&\cong& H_c^{n+1}(\mathbb G_{m, \bar k}\times \mathbb A_{\bar k}^n, 
\pi_1^*\mathcal K_\chi\otimes G_b^*\mathcal L_\psi)[1],\\
\mathrm{dim}\, \mathfrak F(j_![t^{-B}]_!\mathcal F_\chi[1])_{\bar b}&=& -(A+B)(d-1)^n.
\end{eqnarray*}
The sheaf $\mathfrak F(j_![t^{-B}]_!\mathcal F_\chi[1])$ is perverse, and its generic rank 
is $-(A+B)(d-1)^n$. The specialization homomorphism
 $$H^{-1}(\mathfrak F(j_![t^{-B}]_!\mathcal F_\chi[1])_{\bar 0})\to H^{-1}(\mathfrak F(j_![t^{-B}]_!\mathcal F_\chi[1])_{\bar\eta})$$
 is injective. It follows that 
 \begin{eqnarray}
 \mathrm{dim}\, H^{-1}(\mathfrak F(j_![t^{-B}]_!\mathcal F_\chi[1])_{\bar 0})\leq (A+B) (d-1)^n.
 \end{eqnarray}
 We have 
 $$H^{-1}(\mathfrak F(j_![t^{-B}]_!\mathcal F_\chi[1]))_{\bar 0}\cong H^1_c(\mathbb G_{m, \bar k}, \mathcal F_\chi).$$ So 
 we have
 $$\mathrm{dim}\,H^1_c(\mathbb G_{m, \bar k}, \mathcal F_\chi) 
\leq (A+B)(d-1)^n.$$
By (\ref{fact2}), for any $b\in \mathbb G_m(\bar k)$, we have
$$H^{-1}(\mathfrak F(j_![t^{-B}]_!\mathcal F_\chi[1])_{\bar b})
\cong H_c^1(\mathbb G_{m, \bar k}, \mathcal F_\chi\otimes [bt^{-B}]^*\mathcal L_\psi)
\cong H_c^{n+1}(\mathbb G_{m, \bar k}\times \mathbb A_{\bar k}^n, \pi_1^*\mathcal K_\chi\otimes G_b^*\mathcal L_\psi)[2].$$
So $H_c^1(\mathbb G_{m,\bar k} , \mathcal F_\chi\otimes [bt^{-B}]^*\mathcal L_\psi)$ is pure of weight $n+1$ by Theorem \ref{dl} (ii). 
Let $\bar j:\mathbb G_m\hookrightarrow
\mathbb P^1$ and $j':\mathbb G_m\hookrightarrow \mathbb P^1-\{0\}$ be the canonical open immersions. 
By Lemma \ref{pure}, we have 
$$\bar j_!(\mathcal F_\chi\otimes [bt^{-B}]^*\mathcal L_\psi)\cong\bar j_*(\mathcal F_\chi\otimes [bt^{-B}]^*\mathcal L_\psi).$$
Since $[bt^{-B}]^*\mathcal L_\psi$ is lisse on $\mathbb P^1-\{0\}$, this implies that 
$$j'_!\mathcal F_\chi\cong j'_*\mathcal F_\chi.$$ We have $H^0(\mathbb P^1_{\bar k}-\{0\}, j'_!\mathcal F_\chi)=0$ since $\mathcal F_\chi$ is lisse on 
$\mathbb G_m$. 
So  $$H^0(\mathbb G_{m,\bar k}, \mathcal F_\chi)\cong 
H^0(\mathbb P^1_{\bar k}-\{0\}, j'_*\mathcal F_\chi)\cong H^0(\mathbb P^1_{\bar k}-\{0\}, j'_!\mathcal F_\chi)=0.$$ Then by the Poincar\'e duality, we have
$H^2_c(\mathbb G_{m,\bar k},\mathcal F_\chi)=0$. Finally $H^0_c(\mathbb G_{m,\bar k},\mathcal F_\chi)=0$ since $\mathbb G_m$ is affine. 

(ii) Follows from (i) and Proposition \ref{deligne}. 

(iii) Follows from (ii), the Grothendieck trace formula, and Deligne's theorem (the Weil conjecture). 
\end{proof} 

\section{Improved degree bound}

In this section, we prove Theorem \ref{III}.  Let
$$G(t_0, \ldots, t_n)=t_0^A f(t_1, \ldots, t_n) +g(t_1, \ldots, t_n)+P_h(1/t_0) \in k[t_0^{\pm}, t_1,..., t_n],$$
where $f(t_1, \ldots, t_n)$ is a polynomial of degree $d$, $g(t_1, \ldots, t_n)$ is a polynomial of degree $e<d$, 
and $P_h(s)$ is a one-variable polynomial of degree $h$. We shall assume that $e\geq hd/(A+h)$, that is, $h \leq eA/(d-e)$.  
For any positive integer $m$, let $k_m$ be the extension of $k$ of degree $m$. 
Define two exponential sums over $k_m$ by 
\begin{eqnarray*}
 S_m(G)&=& \sum_{t_0\in k_m^*}\sum_{t_1,\ldots, t_n\in k_m} \psi\Big({\rm Tr}_{k_m/k}(G(t_0, \ldots, t_n))\Big),\\
 S_m^*(G)&=& \sum_{t_0\in k_m^*}\sum_{t_1,\ldots, t_n\in k_m^*} \psi\Big({\rm Tr}_{k_m/k}(G(t_0, \ldots, t_n))\Big).
 \end{eqnarray*}
 Their corresponding  $L$-functions are defined by 
$$L(G, T) =\exp\Big(\sum_{m=1}^{\infty}S_m(G)\frac{T^m}{m}\Big), \quad L^*(G, T) =\exp\Big(\sum_{m=1}^{\infty}S_m^*(G)\frac{T^m}{m}\Big).$$
They are rational functions. The degree $\deg(L(G,T))$ of $L(G,T)$ is defined to be the degree of its numerator minus the degree of 
its denominator. 
For each subset $S\subseteq \{1,\cdots, n\}$, let $G_S$ denote the polynomial obtained from $G$ by setting all $t_i=0$ for $i\in \{1,\ldots, n\}- S$. 
Thus, $G_S$ is a Laurent polynomial in $1+|S|$ variables. 
In a similar way, one defines the exponential sum $S_m(G_S)$ over ${\mathbb G}_m \times {\mathbb A}^{|S|}$, 
the exponential sum $S_m^*(G_S)$ over  ${\mathbb G}_m^{1+|S|}$, and their $L$-functions $L(G_S, T)$ and $L^*(G_S, T)$. 
The toric decomposition of ${\mathbb A}^{n}$ gives the decomposition 
$$S_m(G) =\sum_{S\subseteq \{1,\cdots, n\}} S_m^*(G_S), \quad L(G, T) =\prod_{S\subseteq \{1,\cdots, n\} }L^*(G_S,T).$$
Let $\Delta_1$ be the simplex in ${\mathbb R}^{n+1}$ with vertices 
$$(0, \ldots, 0), (A,0,\ldots, 0), (A,d,0\ldots, 0),\ldots, (A, 0,\ldots,0, d),$$ 
let $\Delta_2$ be the simplex in ${\mathbb R}^{n+1}$ with vertices 
$$(-h,0,\ldots, 0),(0, \ldots, 0),  (0, e,0,\ldots, 0), \ldots, (0, 0,\ldots, e),$$
let $\Delta_3$ be the convex hull in ${\mathbb R}^{n+1}$ of the points
$$(0, \ldots, 0), (A,d,0,\ldots, 0), \ldots, (A, 0,\ldots, 0,d), (0, e,0,\ldots, 0), \ldots, (0, 0,\ldots,0,e),$$
and let $\Delta=\Delta_1\cup \Delta_2\cup \Delta_3$. Then $\Delta$ is a convex polytope in ${\mathbb R}^{n+1}$, and the 
Newton polytope at $\infty$ of the Laurent polynomial $G(t_0, \ldots, t_n)$ is contained in $\Delta$. 
By the degree bound of Adolphson-Sperber \cite{AS}, we have 
$$|\deg(L^*(G,T))| \leq (n+1)! {\rm vol}(\Delta)= (n+1)! ({\rm vol}(\Delta_1)+ {\rm vol}(\Delta_2)+ {\rm vol}(\Delta_3)).$$
$\Delta_1$ and $\Delta_2$ are simplexes, and we have
$$(n+1)! {\rm vol}(\Delta_1) = Ad^n, \quad (n+1)! {\rm vol}(\Delta_2) = he^n.$$
The polytope $\Delta_3$ is not a simplex. Since $h \leq eA/(d-e)$, we can write $\Delta_3=\Delta_4 \backslash \Delta_5$, where 
$\Delta_4$ is the simplex in ${\mathbb R}^{n+1}$ with vertices 
$$(0, \ldots, 0), (-\frac{eA}{d-e},0,\ldots,0),  (A,d,0,\ldots, 0), \ldots, (A, 0,\ldots, 0,d),$$
and $\Delta_5$ is the simplex  in ${\mathbb R}^{n+1}$ with vertices 
$$(0,\ldots, 0), (-\frac{eA}{d-e},0,\ldots,0), (0, e,0,\ldots, 0), \ldots, (0, 0,\ldots, 0,e).$$
We have $\Delta_5 \subseteq \Delta_4$. It follows that 
\begin{eqnarray*}
(n+1)! {\rm vol}(\Delta_3) &=& (n+1)! {\rm vol}(\Delta_4) -(n+1)! {\rm vol}(\Delta_5)\\
&=& \frac{eA}{d-e}d^n - \frac{eA}{d-e}e^n.
\end{eqnarray*}
Here to calculate the volume of a simplex, we use the formula that if $\Sigma$ is a simplex in $\mathbb R^n$ with 
vertices $\{e_0, \ldots, e_n\}$, then $n!\mathrm{vol}(\Sigma)=\vert\mathrm{det}(A)\vert$, where $A$ is the matrix 
whose $i$-th row is given by the coordinates of $e_i-e_0$. 
Putting together, we obtain 
$$|\deg(L^*(G,T))| \leq \Big(A+ \frac{eA}{d-e}\Big)d^n -\Big (\frac{eA}{d-e}-h\Big)e^n.$$
Similarly, for each subset $S\subseteq \{1,\cdots, n\}$, we have 
$$|\deg(L^*(G_S,T))| \leq \Big(A+ \frac{eA}{d-e}\Big)d^{|S|} -\Big (\frac{eA}{d-e}-h\Big)e^{|S|}.$$
It follows that 
\begin{eqnarray*}
|\deg(L(G, T))|& \leq& \sum_{S\subseteq \{1,\cdots, n\} }|\deg(L^*(G_S,T))| \\
&\leq& \Big (A+ \frac{eA}{d-e}\Big)\sum_{S\subseteq \{1,\cdots, n\} } d^{|S|} - \Big(\frac{eA}{d-e}-h\Big)\sum_{S\subseteq \{1,\cdots, n\} }e^{|S|} \\
&=& \Big (A+ \frac{eA}{d-e}\Big)(d+1)^{n} - \Big(\frac{eA}{d-e}-h\Big)(e+1)^{n}.
\end{eqnarray*}
Suppose furthermore that $f$ is affine Dwork regular of degree $d$. We can choose a large $B$ not divisible by $p$ so that 
$e<\frac{Bd}{A+B}$. By Corollary \ref{K} (under the condition (b)), we have $$H_c^i(\mathbb G_{m,\bar k}\times \mathbb A_{\bar k}^n, 
\pi_1^*\mathcal K_\chi\otimes G^*\mathcal L_\psi)=0$$ for $i\not=n+1$. By Grothendieck's product formula for $L$-functions, 
$L(G, T)^{(-1)^n}$ is a polynomial whose degree is equal to $\mathrm{dim}\, H_c^{n+1}(\mathbb G_{m,\bar k}\times \mathbb A_{\bar k}^n, 
\pi_1^*\mathcal K_\chi\otimes G^*\mathcal L_\psi)$. So we have 
$$\mathrm{dim}\, H_c^{n+1}(\mathbb G_{m,\bar k}\times \mathbb A_{\bar k}^n, 
\pi_1^*\mathcal K_\chi\otimes G^*\mathcal L_\psi)\leq   \Big (A+ \frac{eA}{d-e}\Big)(d+1)^{n} - \Big(\frac{eA}{d-e}-h\Big)(e+1)^{n}.
 $$
The estimates in Theorem \ref{III} then follows from Grothendieck's trace formula and Deligne's theorem. 

\section{Generic Newton polygon}

Let $d, A, B$ be positive integers relatively prime to $p$. Consider the universal family 
of $(A,B)$-polynomials of the form 
\begin{equation}\label{eq}
G(t):=t_0^A f(t_1, \ldots, t_n) +g(t_1, \ldots, t_n)+P_B(1/t_0) \in \bar{k}[t_0^{\pm}, t_1,..., t_n]
\end{equation}
where $f(t_1, \ldots, t_n)$ is a polynomial of degree $d$, $g(t_1, \ldots, t_n)$ is a polynomial of degree $< Ad/(A+B)$, 
and $P_B(s)$ is a one-variable polynomial of degree exactly $B$. Let $M(d, A,B, p)$ be the Zariski open dense subspace of 
such $(A,B)$-polynomials $G(t)$ satisfying the additional condition that $f$ is affine Dwork regular. 
It is non-empty as the polynomial $f(t)=1+t_1^d+\cdots +t_n^d$ is affine Dwork regular. 

Suppose $G(t) \in M(d,A,B,p)$.  It is non-degenerate with respect to its Newton polytope $\Delta$ at $\infty$, which is the simplex ${\mathbb R}^{n+1}$ with vertices 
$$(-B, 0, \ldots, 0), (A,0,\ldots, 0), (A,d,0,\ldots, 0),\ldots, (A, 0,\ldots,0, d).$$  
By the work of Adolphson-Sperber \cite{AS2}, the $L$-function $L^*(G(t), T)^{(-1)^n}$ for the exponential sum over the torus 
$\mathbb{G}_m^{n+1}$ is a polynomial of degree $(A+B)d^n$, mixed of weights $\leq n+1$. 
Its Newton polygon lies above 
certain combinatorically defined lower bound called the Hodge polygon $\mathrm{HP}^*(\Delta)$. 

By the Grothendieck specialization theorem, the Newton polygon goes up under specialization. The generic Newton polygon exists for the family of 
$(A,B)$-polynomials $G(t) \in M(d,A,B,p)$. It is just the lowest possible Newton polygon as $G(t)$ varies 
in $M(d,A,B,p)$. Denote this generic Newton polygon by $\mathrm{GNP}^*(d, A,B,p)$, 
which also lies above $\mathrm{HP}^*(\Delta)$. If the two polygons coincide, we say that the family $M(d,A,B,p)$ 
is \emph{generically ordinary} for its $L$-function over the torus $\mathbb{G}_m^{n+1}$. This property depends only on the four numbers $d,A,B,p$. 

Similarly, for $G(t) \in M(d,A,B,p)$, the $L$-function 
$L(G(t), T)^{(-1)^n}$ for the exponential sum over $\mathbb{G}_m \times \mathbb{A}^n$ 
is a polynomial of degree $(A+B)(d-1)^n$, pure of weight $n+1$. Its Newton polygon lies above 
certain combinatorically defined lower bound called the Hodge polygon $\rm{HP}(\Delta)$. 
The generic Newton polygon is denoted by $\mathrm{GNP}(d, A,B,p)$, 
which also lies above $\mathrm{HP}(\Delta)$. If the two polygon coincides, we say that the family $M(d,A,B,p)$ 
is generically ordinary for its $L$-function over $\mathbb{G}_m \times \mathbb{A}^n$. This property again depends only on the four numbers $d,A,B,p$. 

Note that our family $M(d,A,B,p)$  can be strictly smaller than the universal family $M(\Delta, p)$ of 
non-degenerate and commode (with respect to $t_1, \ldots, t_n$) Laurent polynomials whose Newton polytope 
at $\infty$ is the given $\Delta$,  
as $\Delta$ may contain some lattice points which do not arise from exponents of the terms in $G(t)$. 
For this reason, proving generic ordinariness for this smaller family $M(d,A,B,p)$ can be 
somewhat harder than that for the larger family $M(\Delta, p)$. 
We prove that this is indeed true if $p \equiv 1 \mod [A, dB]$, where $[A, dB]$ denotes the 
least common multiple. That is, we have 

\begin{theorem}\label{NP} If $p \equiv 1 \mod [A, dB]$, then we have 
$$\mathrm{GNP}^*(d, A,B,p)= \mathrm{HP}^*(\Delta),\quad  \mathrm{GNP} (d, A,B,p)= \mathrm{HP}(\Delta).$$
\end{theorem}

\begin{proof}
The first assertion is stronger. It implies the second assertion, as the second assertion is a portion of the first assertion by the 
boundary decomposition theorem in \cite[Section 5]{W1}. 
To prove the first assertion, 
 we apply the various decomposition theorems in \cite{W1, W2}.  
 
 Let $\Delta_1$ be the simplex in ${\mathbb R}^{n+1}$ with vertices 
$$(0, \ldots, 0), (A,0,\ldots, 0), (A,d,0\ldots, 0),\ldots, (A, 0,\ldots,0, d).$$ 
Let $\Delta_2$ be the simplex in ${\mathbb R}^{n+1}$ with vertices 
$$(-B,0,\ldots, 0),(0, \ldots, 0),  (A, d,0,\ldots, 0), \ldots, (A, 0,\ldots, 0,d).$$
It is clear that $\Delta$ is the union of $\Delta_1$ and $\Delta_2$. This is the facial decomposition (\cite[Section 5]{W1}) of $\Delta$.  
The restriction of our universal family $G(t)$ to the unique codimension $1$ face of $\Delta_1$ 
not containing the origin is the following 
family of polynomials 
$$G_1(t) = t_0^A f(t_1,..., t_n),$$
where $f$ is a polynomial of degree $d$.  The Newton polytope at $\infty$ of this family is precisely $\Delta_1$. 
The restriction of our universal family $G(t)$ to the unique codimension $1$ face of $\Delta_2$ 
not containing the origin is the following 
family of Laurent polynomials 
$$G_2(t) = t_0^A f_d(t_1,..., t_n) +bt_0^{-B},$$
where $f_d$ is the leading form of $f$, and $b$ is the leading coefficient of $P_B$. The Newton 
polytope at $\infty$ of this family is precisely $\Delta_2$. 

The facial decomposition theorem in \cite[Theorem 5.5]{W1} says that $G(t)$ is ordinary with respect to $\mathrm{HP}^*(\Delta)$ if and only if 
$G_i(t)$ is 
ordinary with respect to $\mathrm{HP}^*(\Delta_i)$ for each $i\in\{1, 2\}$. 
In particular, for the ordinary property, as long as $\deg(g) < Ad/(A+B)$, the polynomial 
$g(t_1,\ldots, t_n)$ plays no role as its exponents do not lie on any of the two codimension $1$ faces 
of $\Delta$ not containing the origin.  Similarly, the lower degree terms in $P_B(t)$ and in $f(t)$ are irrelevant as far as the 
ordinariness property is concerned. 

Now, the first family $G_1(t)$ is generically ordinary with respect to $\mathrm{HP}^*(\Delta_1)$ under the condition $p \equiv 1 \mod A$. 
This is proved in \cite[Theorem 7.5]{W1} using a sequence of parallel hyperplane decompositions. 
In the special case $A=1$, it implies that the zeta function of the 
universal family of toric (or affine or projective) hypersurfaces of degree $d$ is  generically 
ordinary for every prime $p$ and every $n$,  a highly nontrivial result already. Using a similar sequence of parallel hyperplane decompositions, 
one deduces that the second family $G_2(t)$ is generically ordinary with
 respect to $\mathrm{HP}^*(\Delta_2)$ under the condition $p \equiv 1 \mod dB$. Putting together, we obtain Theorem \ref{NP}. 
\end{proof}

\begin{remark}  
As indicated above, both $\Delta_1$ and $\Delta_2$ are simplexes. Instead of using the hyperplane decomposition 
theorem, an alternative easier approach is  to choose an elementary diagonal example (the number of nonzero terms equals  
the number of variables) for each family $G_i(t)$ ($1\leq i\leq 2$) to compute its Newton polygon. 
A diagonal example in the family $G_1(t)$ is the non-degenerate polynomial 
$$G^{(0)}_1(t)=t_0^A (1 + t_1^d +\cdots + t_n^d).$$
The matrix of its exponents is a square matrix with the largest invariant factor $[d, A]$ 
and thus $G^{(0)}_1(t)$ is ordinary if $p \equiv 1 \mod [d, A]$ by \cite[Corollay 2.6]{W2}. This gives a weaker result 
for the first family than 
what can be obtained by using the hyperplane decomposition, but it gives the same result 
as Theorem \ref{NP}. 
Similarly, a diagonal example in the family $G_2(t)$ is the non-degenerate polynomial 
$$G^{(0)}_2(t) = t_0^A (t_1^d +\cdots + t_n^d) + t_0^{-B}.$$
The matrix of its exponents is a square matrix with the largest invariant factor dividing $dB$ 
and thus $G^{(0)}_2(t)$ is ordinary if $p \equiv 1 \mod  dB$ by \cite[Corollay 2.6]{W2}. It follows that the total 
family $G(t)$ is generically ordinary if $p \equiv 1 \mod [A, dB]$. 
\end{remark}

\begin{remark}  
Given an integral $(n+1)$-dimensional convex polytope in $\mathbb{R}^{n+1}$ 
containing the origin, we can consider the universal family $M(\Delta, p)$ of all non-degenerate Laurent polynomials whose Newton polytope 
at $\infty$ is $\Delta$. Its 
generic Newton polygon over the torus $\mathbb{G}_m^{n+1}$ is denoted by $\mathrm{GNP}^*(\Delta,p)$,  which depends only on $\Delta$ and $p$. 
Adolphson-Sperber's work implies that $\mathrm{GNP}^*(\Delta,p)$ lies above  a certain explicit lower bound $\mathrm{HP}^*(\Delta)$, called the Hodge polygon.  
They conjectured \cite{AS2} that 
$$\mathrm{GNP}^*(\Delta,p) = \mathrm{HP}^*(\Delta) \hbox{ if } p \equiv 1 \mod D(\Delta),$$
where $D(\Delta)$ is the denominator of $\Delta$. It is not hard to prove that the condition $p \equiv 1 \mod D(\Delta)$ 
is necessary (and thus optimal) for the conjecture to be true, either by a direct combinatorial proof or by a ramification argument. 
The Adolphson-Sperber conjecture is false in general, but true in many importance cases as shown in \cite{W1}\cite{W2},  including 
notably the above $\Delta_1$ coming from the first family $G_1(t)$. 
We expect that this conjecture is true for the above $\Delta$ defined by the $(A,B)$-polynomial. One checks that 
$D(\Delta_1)=A$ since the unique codimension $1$ face of $\Delta_1$ not containing the origin is defined by the hyperplane equation $\frac{1}{A}t_0=1$. 
Similarly, one checks that $D(\Delta_2)= [B, dB/(A+B, dB)]$ since the unique codimension $1$ face of $\Delta_2$ not containing the origin is defined by the hyperplane equation 
$$\frac{-1}{B}t_0 + \frac{A+B}{dB} t_1+\cdots + \frac{A+B}{dB} t_n =1.$$
It follows that 
$$D(\Delta)= [D(\Delta_1), D(\Delta_2)]  = [A, B, dB/(A+B, dB)].$$
For the $\Delta$ defined using the $(A,B)$-polynomials, the Adolphson-Sperber conjecture 
says that $\mathrm{GNP}^*(\Delta,p) = \mathrm{HP}^*(\Delta)$ if $p \equiv 1 \mod [A, B, dB/(A+B, dB)]$. 
We expect this to be true. It is sufficient to prove it for the second piece $\Delta_2$ as the first piece $\Delta_1$ 
is already known as seen above. However,  we do not expect that the condition $p \equiv 1 \mod [A, dB]$ in Theorem \ref{NP} 
can be relaxed to $p \equiv 1 \mod [A, B, dB/(A+B, dB)]$, as the family $M(d,A,B,p)$ can be significantly smaller 
than the family $M(\Delta, p)$. 
\end{remark}

Although the recursive combinatorial definition of the Hodge numbers in $\mathrm{HP}^*(\Delta)$ and $\mathrm{HP}(\Delta)$ 
as given in \cite{AS2} are not 
complicated, a simple explicit formula for the Hodge numbers can be cumbersome to obtain. To give an indication of
what the generic slopes  look like, we give, without proof, an explicit formula for the Hodge numbers and thus the Hodge polygon 
for our $(A,B)$-polytope $\Delta$. For simplicity of notations, we shall assume that $A=B=1$. 

Define 
\begin{eqnarray*}
H^*(T)& =& \prod_{0\leq j_1, \ldots, j_n \leq d-1} \Big(1- q^{\frac{j_1}{d}+\cdots + 
\frac{j_n}{d} +\{\frac{j_1}{d}+\cdots + \frac{j_n}{d}\} }T\Big)\Big(1- q^{\frac{j_1}{d}+\cdots + 
\frac{j_n}{d} +1-\{\frac{j_1}{d}+\cdots + \frac{j_n}{d}\} }T\Big),\\
H(T) &= &\prod_{1\leq j_1, \ldots, j_n \leq d-1}\Big (1- q^{\frac{j_1}{d}+\cdots + \frac{j_n}{d} 
+\{\frac{j_1}{d}+\cdots + \frac{j_n}{d}\} }T\Big)\Big(1- q^{\frac{j_1}{d}+\cdots + \frac{j_n}{d} +1-\{\frac{j_1}{d}+\cdots + \frac{j_n}{d}\} }T\Big),
\end{eqnarray*}
where $\{r\}=r-[r]$ denotes the 
fractional part of $r$.  
It is clear that $H^*(T)$ is a polynomial of degree $(A+B)d^n =2d^n$. 
Similarly, $H(T)$ is a polynomial of degree $(A+B)(d-1)^n =2(d-1)^n$, whose slopes are symmetric  in the 
interval $[0, n+1]$. In the case $A=B=1$, the Hodge polygon $\mathrm{HP}^*(\Delta)$ (resp. $\mathrm{HP}(\Delta)$)
is simply  
the $q$-adic Newton polygon of $H^*(T)$ (resp. $H(T)$). Note that the slopes of $H^*(T)$ and $H(T)$ are rational numbers 
with denominators dividing $d$. The coefficients of the $L$-function lie in the $p$-th cyclotimic field $\mathbb{Q}(\zeta_p)$ 
which is totally ramified of degree $p-1$ over $p$. This explains the congruence condition $p\equiv 1 \mod d$ 
of Theorem \ref{NP} in the case $A=B=1$. 

Finally, if our $(A,B)$-exponential sum is twisted by a multiplicative character $\chi$ of order $m$ dividing $q -1$, 
then the generic Newton polygon for the corresponding twisted $L$-function over the torus $\mathbb{G}_m^{n+1}$  
(resp., over $\mathbb{G}_m \times \mathbb{A}^n$) lies  above $\mathrm{HP}^*(\Delta)$ (resp. over $\mathrm{HP}(\Delta)$). 
Furthermore, these two polygons coincide if $p \equiv 1 \mod [mA, dmB]$. One simply applies the 
above theorem to the $(mA, mB)$-polynomial $G(t_0^m, t_1,\ldots, t_n)$ and decomposes in terms of 
the multiplicative characters $\chi$ of order dividing $m$.  

\vfill
\eject

\end{document}